\documentclass[oneside,a4paper,12pt,notitlepage]{article}
\usepackage[left=0.8in, right=0.8in, top=1.5in, bottom=1.5in]{geometry}
\usepackage[T1]{fontenc} 
\usepackage[utf8]{inputenc}
\usepackage{lipsum}
\usepackage{lmodern}
\usepackage{amssymb}
\usepackage{amsthm}
\usepackage{bm}
\usepackage{mathtools}
\usepackage{braket}
\usepackage{esint}
\newcommand{\abs}[1]{{\left|#1\right|}}
\newcommand{\norma}[1]{{\left\Vert#1\right\Vert}}

\usepackage{booktabs}
\usepackage{graphicx}
\usepackage{tikz}
\usetikzlibrary{patterns, math, intersections}
\usepackage{multicol}
\usepackage{caption}
\usepackage{enumerate}
\usepackage[skins,theorems]{tcolorbox}
\tcbset{highlight math style={enhanced,
		colframe=black,colback=white,arc=0pt,boxrule=1pt}}
\def\XXint#1#2#3{{\setbox0=\hbox{$#1{#2#3}{\int}$}
    \vcenter{\hbox{$#2#3$}}\kern-.5\wd0}}

\theoremstyle{definition}

\theoremstyle{plain}
\newtheorem{teorema}{Theorem}[section]

\newtheorem{lemma}[teorema]{Lemma}
\newtheorem{prop}[teorema]{Proposition}

\theoremstyle{definition}
\newtheorem{esempio}{Example}[section]
\newtheorem{oss}[esempio]{Remark}

\renewcommand{\div}{\text{div}}

\DeclareMathOperator{\R}{\mathbb{R}}

\DeclareMathOperator{\diam}{\text{diam}}

\makeatletter
\newcommand{\myfootnote}[2]{\begingroup
	\def\@makefnmark{}%
	\addtocounter{footnote}{-1}%
	\footnote{\textbf{#1} #2}%
	\endgroup}
\makeatother

\newcommand{\meno}{\setminus}
\usepackage{comment}

\newcommand{\lnorma}[1]{{\left\vert\kern-0.25ex\left\vert\kern-0.25ex\left\vert #1 \right\vert\kern-0.25ex\right\vert\kern-0.25ex\right\vert}}
\theoremstyle{definition}
\newtheorem{open problem}{Open Problem}

\numberwithin{equation}{section}

\definecolor{brightpink}{rgb}{1.0, 0.0, 0.5}

\usepackage{todonotes}

\usepackage{hyperref}
\hypersetup{linktoc=none, bookmarksnumbered, colorlinks=true, linkcolor=red}

\title{Estimates for Robin $p$-Laplacian eigenvalues of convex sets with prescribed perimeter}
\author{Vincenzo Amato, Andrea Gentile, Alba Lia Masiello}
\date{\today}
\begin{document}
	\maketitle
	\begin{abstract} 
	In this paper, we prove an upper bound for the first Robin eigenvalue of the $p$-Laplacian with a positive boundary parameter and a quantitative version of the reverse Faber-Krahn type inequality for the first Robin eigenvalue of the $p$-Laplacian with negative boundary parameter, among convex sets with prescribed perimeter. 
	
	The proofs are based on a comparison argument obtained by means of inner sets, introduced by Payne, Weimberger \cite{paywie} and Polya \cite{polya1960}.
	

\end{abstract}
\section{Introduction}

Let $\Omega$ be a bounded, open and convex set in $\R^n$. We consider the following problem
\begin{equation}
    \label{bpos}
    \begin{cases}
        -\Delta_p u=-\div(\abs{\nabla u}^{p-2}\nabla u)= \lambda_{p,\beta}(\Omega) \abs{u}^{p-2} u \, & \text{in} \, \, \Omega\\
        \displaystyle{\abs{{\nabla u}}^{p-2}\frac{\partial u}{\partial \nu}+ \beta \abs{u}^{p-2} u=0} \, & \text{on} \, \, \partial\Omega,
    \end{cases}
\end{equation}
where $\beta\in \R$.

The fundamental eigenvalue of the Robin Laplacian on $\Omega$ is defined by
\begin{equation}
    \label{rel}
    \lambda_{p,\beta}(\Omega)=\min_{\substack{v\in W^{1,p}(\Omega) \\v\neq 0}} \frac{\displaystyle{\int_\Omega \abs{\nabla v}^p\, dx+\beta \int_{\partial\Omega} \abs{v}^p\, d\mathcal{H}^{n-1}}}{\displaystyle{\int_\Omega \abs{v}^p\, dx}}
\end{equation}
and a minimizer $u$ in \eqref{rel} satisfies the equation \eqref{bpos} in the weak form if

\begin{equation*}
  \int _\Omega \abs{\nabla u}^{p-2} \nabla u \nabla \varphi \, dx+ \beta \int_{\partial\Omega}\abs{u}^{p-2}u\varphi \, d \mathcal{H}^{n-1} = \lambda_{p,\beta}(\Omega) \int_{\Omega} \abs{u}^{p-2} u \varphi \, dx, \quad \forall \varphi \in W^{1,p}(\Omega).
\end{equation*}
It is well-known (see for instance \cite{Lindqvist_on_the_equation} for the Dirichlet case) that the first eigenvalue \eqref{rel} is simple and that in the case of a ball, the corresponding eigenfunction is radially symmetric (\cite{Bhattacharya}).

In this paper we want to study the different behaviour of eigenvalues in the case $\beta>0$ and $\beta<0$. For sake of completeness, we recall that, in the case of $ \beta = 0 $, we recover Neumann boundary condition, for which the first eigenvalue is zero and the associated eigenfunctions are constants.

It is well known that in the case $\beta>0$, Bossel \cite{Bossel198647} and Daners \cite{Daners} proved a Faber-Krahn inequality for the first eigenvalue of the Robin-Laplacian in two and higher dimensional case, respectively. In particular, they proved that among sets of given volume, the one which minimizes the first Robin-eigenvalue is the ball, i.e.
\begin{equation}
    \label{Bossel_Daners}
    \lambda_{2,\beta}(\Omega^\sharp)\le \lambda_{2,\beta}(\Omega),
\end{equation}
where $\Omega^\sharp$ is the ball, centered at the origin, having the same volume as $\Omega$. This result was generalized by Bucur, Daners and Giacomini in \cite{bucdan, bg1, bg2} to the eigenvalues of the $p$-Laplacian with Robin boundary conditions. 

This Faber-Krahn inequality for fixed volume and the following rescaling property \cite{Bucur_Freitas_Kennedy}
\begin{equation}
    \label{risc}
    \lambda_{p,\beta}(t\Omega)\le \frac{1}{t}\lambda_{p,\beta}(\Omega)\leq\lambda_{p,\beta}(\Omega),  \quad \forall t>1,
\end{equation}
give a Faber-Krahn inequality for fixed perimeters, i.e.
\begin{equation*}
\lambda_{p,\beta}(\Omega^\star)\leq\lambda_{p,\beta}(\Omega),
\end{equation*}
where $\Omega^\star$ is the ball having the same perimeter as $\Omega$.

Our aim is to give a continuity bound to the ratio
\[
\frac{\lambda_{p,\beta}(\Omega)-\lambda_{p,\beta}(\Omega^\star)}{\lambda_{p,\beta}(\Omega)},
\]
indeed we prove

\begin{teorema}
    \label{teorema_1}
    Let $\beta$ be a positive parameter. Let $\Omega$ be a bounded, open and convex set in $\R^n$ and let $\Omega^\star$ be the ball, centered at the origin, such that $P(\Omega)=P(\Omega^\star)=\rho$. Let $\lambda_{p,\beta}(\Omega)$ and $\lambda_{p,\beta}(\Omega^\star)$ be the first eigenvalues of the $p$-Laplacian operator with Robin boundary conditions on $\Omega$ and $\Omega^\star$, let $v$ be a positive eigenfunction associated to $\lambda_{p,\beta}(\Omega^\star)$. Then
    \begin{equation}
        \label{beta_pos}
        \frac{\lambda_{p,\beta}(\Omega)-\lambda_{p,\beta}(\Omega^\star)}{\lambda_{p,\beta}(\Omega)} \leq C(n,p,\beta, \rho)\left(1- \frac{n^{\frac{n}{n-1}}\omega_n^{\frac{1}{n-1}}\abs{\Omega}}{P(\Omega)^{\frac{n}{n-1}}}\right),
    \end{equation}
    where $\omega_n$ is the measure of the unitary ball in $\R^n$, and $\displaystyle{C(n,p,\beta,\rho)= \frac{\norma{v}_\infty^p\abs{\Omega^\star}}{\norma{v}^p_p}}$.
\end{teorema}

 It is possible to give a uniform bound to the constant in \eqref{beta_pos} from above with a constant independent of the parameter $\beta$ and the perimeter, indeed it holds 
 
\begin{equation}\label{conun}
    C(n,p,\beta,\rho)\le C(n,p):= \frac{\norma{v_\infty}_\infty^p\abs{\Omega^\star}}{\norma{v_\infty}^p_p},
\end{equation}  
 where $v_\infty$ is the first Dirichlet eigenfunction. Thanks to the rescaling property of Dirichlet eigenvalues and eigenfunctions, $C(n,p)$ does not depends on the perimeter. For more details see Remark \ref{rem1}.

 We observe that this result can be seen as a generalization to the Robin case of the result in \cite{BNT}, which holds true in the case of Dirichlet eigenvalues of the $p$-Laplacian.

 When $\beta$ is a negative parameter, the authors in \cite{BFNT} proved a reverse Faber-Krahn inequality for the first eigenvalue of the Dirichlet-Laplacian among convex sets of given perimeter. In particular, they proved that among convex sets of given perimeter the ball $\Omega^\star$ maximizes the first Robin eigenvalue of the $p$-Laplacian, i.e.
 \begin{equation}\label{refb}
     \lambda_{p,\beta}(\Omega)\le\lambda_{p,\beta}(\Omega^\star).
 \end{equation}
 For completeness' sake, we quote that in \cite{AFK}, the authors already proved that the disc maximizes the first eigenvalue under a perimeter constraint, among $C^2$ domains in $\R^2$, while the question remained open in arbitrary dimension.
 
This question is related to the conjecture of Bareket (see \cite{Bareket}) claiming that the ball maximizes $\lambda_{p,\beta}(\Omega)$ among all Lipschitz sets with given volume. Freitas and Krejčiřík in \cite{Freitas_Krejcirik} proved that the conjecture is false, giving a counter-example based on the asymptotic behaviour of the eigenvalues on a disc and an annulus of the same area when $\beta\to-\infty$. They also proved that among sets of area equal to $1$, the conjecture is true, provided $\beta$ is close to $0$.

In \cite{CL} the authors proved a quantitative version of the reverse Faber-Krahn \eqref{refb} in the case $p=2$ among convex sets of fixed perimeter following the Fuglede's approach introduced in \cite{Fuglede}.

In this paper, we recover the result in \cite{CL} obtaining a quantitative version of \eqref{refb} for all $p$, but using a different approach. Indeed, following the method introduced by Payne and Weinberger in \cite{paywie}, we establish a comparison using the so-called parallel coordinates method. In particular, we  firstly prove a lower bound in terms of perimeter and measure of $\Omega$, that is

\begin{teorema}
   \label{teo_beta_negative}
    Let $\beta$ be a negative parameter. Let $\Omega$ be a bounded, open and convex set in $\R^n$ and let $\Omega^\star$ be the ball, centered at the origin, such that $P(\Omega)=P(\Omega^\star)=\rho$. Let $\lambda_{p,\beta}(\Omega)$ and $\lambda_{p,\beta}(\Omega^\star)$ be the first eigenvalues of the $p$-Laplacian operator with Robin boundary conditions on $\Omega$ and $\Omega^\star$, let $v$ be a positive eigenfunction associated to $\lambda_{p,\beta}(\Omega^\star)$. Then
    \begin{equation}
        \label{beta_neg}
        \frac{\lambda_{p,\beta}(\Omega^\star) - \lambda_{p,\beta}(\Omega)}{\abs{\lambda_{p,\beta}(\Omega)}}
        \geq C(n,p,\beta,\rho) \left(1 - \frac{n^{\frac{n}{n-1}}\omega_n^{\frac{1}{n-1}}\abs{\Omega}}{P(\Omega)^{\frac{n}{n-1}}}\right),
    \end{equation}
    where $\omega_n$ is the measure of the unitary ball in $\R^n$, and $\displaystyle{C(n,p,\beta,\rho)= \frac{v_m^p \abs{\Omega^\star}}{\norma{v}^p_p}}$ with $\displaystyle{v_m =\min_{\Omega^\star} v}$.
\end{teorema}

In this case, the constant $C(n,p,\beta,\rho)$ cannot be replaced by a constant independent of $\beta$ and of the perimeter, as it is shown in Remark \ref{modbes}.

Then we prove the quantitative result as in \cite{CL}.
\begin{teorema}
    \label{teo_quantitativa_con_Hausdorff}
    Let $n\geq 2$, $\rho > 0$ and $\beta< 0$. Then there exists two positive constants $C(n,p,\beta,\rho) >0$ and $\delta_0(n,p,\beta,\rho) >0$, such that, for all $\Omega \subset \R^n$ bounded and convex with $P(\Omega) = \rho$ and $\lambda_{p,\beta}(\Omega^\star)-\lambda_{p,\beta}(\Omega) \leq \delta_0$, it holds 
    \begin{equation}
        \lambda_{p,\beta}(\Omega^\star)-\lambda_{p,\beta}(\Omega) \geq C(n,p,\beta , \rho) g(\mathcal{A}_{\mathcal{H}}^{\star}(\Omega))
    \end{equation}
    where $\Omega^\star$ is a ball with the same perimeter of $\Omega$, $\mathcal{A}_{\mathcal{H}}^{\star}$ is the Hausdorff asymmetry defined in \eqref{deff_A_H_star} and $g$ is defined in \eqref{function_g}.
\end{teorema}

The paper is organized as follows: in Section \ref{section_notion} we recall some preliminary results and useful tools for our aim; in Section \ref{sec3} we provide the proof of our main results.

\section{Notations and Preliminaries}
\label{section_notion}
Throughout this article, $|\cdot|$ will denote the Euclidean norm in $\mathbb{R}^n$,
while $\cdot$ is the standard Euclidean scalar product for  $n\geq2$. By $\mathcal{H}^k(\cdot)$, for $k\in [0,n)$, we denote the $k$-dimensional Hausdorff measure in $\mathbb{R}^n$.
 
The measure and the perimeter of $\Omega$ in $\mathbb{R}^n$ will be denoted by $\abs{\Omega}$ and  $P(\Omega)$, respectively, and, if $P(\Omega)<\infty$, we say that $\Omega$ is a set of finite perimeter. In our case, $\Omega$ is a bounded, open and convex set; this ensures us that $\Omega$ is a set of finite perimeter and that $P(\Omega)=\mathcal{H}^{n-1}(\partial\Omega)$. Moreover, if $\Omega$ is an open set with Lipschitz boundary, it holds

\begin{teorema}[Coarea formula]
    Let $f:\Omega\to\R$ be a Lipschitz function and let $u:\Omega\to\R$ be a measurable function. Then,
    \begin{equation}
        \label{coarea}
        \int_{\Omega} u(x) \lvert \nabla f(x) \rvert \, dx=\int _{\mathbb {R} }dt\int_{(\Omega\cap f^{-1}(t))}u(y)\, d\mathcal{H}^{n-1}(y).
    \end{equation}
\end{teorema}
Some references for results relative to the sets of finite perimeter and for the coarea formula are, for instance \cite{AFP, maggi_sets_finite_perimeter}.

\subsection{Asymmetry index $\mathcal{A}^\star(E)$}
In the case $\beta<0$, we want to prove a quantitative result  so we need an index that gives us information about the shape of $\Omega$. We will consider the Hausdorff asymmetry index, as already done in \cite{CL}, so we recall some basic notions about the Hausdorff distance. 

We recall that if $E,F$ are any two convex sets in  $\mathbb{R}^n$ the Hausdorff distance between $E, F$ is defined as
\[
d_{\mathcal{H}}(E,F) := \inf\left\{ \varepsilon>0 \, : \,  E \subset F + \varepsilon B , \, F \subset E + \varepsilon B  \right\},
\]
where $B$ is the unitary ball centered at the origin and $F + \varepsilon B$ is the well-known Minkowski sum.
For such set we define two isoperimetric deficit 
\begin{equation}
    \mathcal{D}(E):= P(E)- P(E^{\sharp}) , \quad\mathcal{M}(E):= \abs{E^\star} - \abs{E}
\end{equation}
where $E^{\sharp}$ and $E^{\star}$ are the ball with the same measure and the same perimeter of $E$, respectively. 

Moreover, we will consider the following Hausdorff asymmetry indices
\begin{equation}
    \label{deff_A_H_star}
    \mathcal{A}_{\mathcal{H}}^\star(E)= \min_{x \in \R^{n}} \left\{d_{\mathcal{H}}(E, B_r(x)) ,\, P(\Omega)= P(B_r(x)) \right\}
\end{equation}
and
\begin{equation}
    \label{deff_A_H_sharp}
    \mathcal{A}_{\mathcal{H}}^\sharp(E)=\min_{x \in \R^{n}} \left\{d_{\mathcal{H}}(E, B_r(x)) ,\, \abs{\Omega}= \abs{B_r(x)} \right\}.
\end{equation}

Lemma 2.9 in \cite{GLPT} tells us how these two indices are related one to the other
\begin{lemma}
    \label{gloria}
    Let $n \geq 2 $ and let $E \subset \R^n$ be a bounded, convex, with $\mathcal{D}(E) \leq \delta$  then 
    \begin{equation}
        \mathcal{A}_H^{\star}(E) \leq C(n)\mathcal{A}_H^{\sharp}(E).
    \end{equation}
\end{lemma}
With these definitions, we can recall the quantitative isoperimetric inequality proved in \cite{Fuglede,Fusco}.

\begin{teorema}[Fuglede]
\label{fuglede}
    Let $n \geq 2$, and let $E$ be a bounded open and convex set with $ \abs{E}= \omega_n$. There exists $\delta, C$, depending only on $n$, such that if $\mathcal{D}(E) \leq \delta$ then 
    then
    \begin{equation}
    \label{eq_Fuglede}
        \mathcal{D}(E) \geq C g \bigl( \mathcal{A}_{\mathcal{H}}^{\sharp}(E) \bigr),
    \end{equation}
    where $g$ is defined by
    \begin{equation}
    \label{function_g}
        g(s)=
        \begin{cases}
            s^2 & \text{ if } n=2\\
            f^{-1}(s^2) & \text{ if } n=3\\
            s^{\frac{n+1}{2}} & \text{ if } n\geq 4
        \end{cases}
    \end{equation}
    and $f(t)= \sqrt{t \log(\frac{1}{t})}$ for $0 < t < e^{-1}$.
\end{teorema}

We are interested in a modified version of this theorem, in terms of $\mathcal{M}(\Omega)$, so we have
\begin{lemma}
    \label{lemma_fugl_mod}
     Let $\Omega \subset \R^n$ be a bounded, open and convex set and let $\Omega^{\star}$ be the ball satisfying $P(\Omega)=P(\Omega^\star)=\rho$. Then, there exist $\delta,C$, depending only on $n$ and $\rho$, such that, if
    \begin{equation}
        \label{Hp_Lemma}
        \mathcal{M}(\Omega) = \abs{\Omega^{\star}} - \abs{\Omega} \leq \delta
    \end{equation}
    then
    \[
    \mathcal{M}(\Omega) \geq C g \bigl( \mathcal{A}_{\mathcal{H}}^{\star}(\Omega) \bigr)
    \]
    where $g$ is the function defined in \eqref{function_g}.
\end{lemma}

\begin{proof}
    Let us start by setting $\delta < \frac{\abs{\Omega^\star}}{2}$, so
    \begin{equation}
        \label{misura_Omega_star_controllata_con_Omega_star}
        \abs{\Omega} > \frac{\abs{\Omega^{\star}}}{2}
    \end{equation}
    and let us assume $\abs{\Omega}= \omega_n$.
    
    By \eqref{misura_Omega_star_controllata_con_Omega_star} and the differentiability of  the function $h(t)=t^{\frac{n-1}{n}}$, we have
    \[
    \abs{\Omega^\star}^{\frac{n-1}{n}}-\abs{\Omega}^{\frac{n-1}{n}}= h'(\xi)  \left(\abs{\Omega^\star}-\abs{\Omega}\right) \leq \frac{n-1}{\omega_n^{\frac{1}{n}} n}\left(\abs{\Omega^\star}-\abs{\Omega}\right).
    \]
    Hence, if $\mathcal{M}(\Omega) \leq \delta$, then
    \[
    \abs{\Omega^\star}^{\frac{n-1}{n}}-\abs{\Omega}^{\frac{n-1}{n}}=\frac{P(\Omega)}{n \omega^{\frac{1}{n}}}-\abs{\Omega}^{\frac{n-1}{n}} = \frac{\mathcal{D}(\Omega)}{n \omega^{\frac{1}{n}}}\leq \tilde{\delta}(n,\rho).
    \]
    So we can apply \eqref{eq_Fuglede}
    \[
    P(\Omega) \geq n \omega_n  \Bigl( 1+\gamma(n) g \bigl( \mathcal{A}_{\mathcal{H}}^{\sharp} (\Omega) \bigr) \Bigr)
    \]
    to obtain
    \begin{align*}
        \abs{\Omega^{\star}} - \abs{\Omega} & = \frac{P(\Omega)^{\frac{n}{n-1}}}{ n^{\frac{n}{n-1}} \omega_n^{\frac{1}{n-1}} } - \abs{\Omega} \\[0.5ex]
        & \geq \omega_n \left(1+\gamma(n) g\bigl( \mathcal{A}_{\mathcal{H}}^{\sharp} (\Omega) \bigr) \right)^{\frac{n}{n-1}} - \omega_n \\[0.5ex]
        & \geq \frac{n \omega_n \gamma(n) }{n-1} g \bigl( \mathcal{A}_{\mathcal{H}}^{\sharp} (\Omega) \bigr),
    \end{align*}
    where in the last step we used Bernoulli's inequality
    \[
    (1+x)^r \geq 1+rx \qquad x \geq -1, \, r\geq 1.
    \]
    Applying Lemma \ref{gloria} we finally have
    \[
    \abs{\Omega^{\star}} - \abs{\Omega} \geq C(n) g \bigl( \mathcal{A}_{\mathcal{H}}^{\star}(\Omega) \bigr).
    \]
    
    The general case  will be recovered by rescaling and by the following inequality for $Y = \omega_n \abs{\Omega}^{-\frac{1}{n}} \Omega$
    \begin{equation}
        \abs{Y^{\star}} - \abs{Y} \leq 2 \omega_n^n \abs{\Omega^\star}^{-1}\left(\abs{\Omega^{\star}} - \abs{\Omega}\right).\qedhere
    \end{equation}
\end{proof}

\subsection{Quermassintegrals}
Let us recall some basic facts about convex sets. Let $K \subset \R^n$ be a non-empty, bounded, convex set, let $B$ be the unitary ball centered at the origin and $\rho > 0$. We can write the Steiner formula for the Minkowski sum $K+ \rho B$ as
\begin{equation}
    \label{Steiner_formula}
    \abs{K + \rho B} = \sum_{i=0}^n \binom{n}{i} W_i(K) \rho^{i} .
\end{equation}
The coefficients $W_i(K)$ are known in literature as quermassintegrals of $K$. In particular, $W_0(K) = \abs{K}, nW_1(K) = P(K)$ and $W_n(K) = \omega_n$ where $\omega_n$ is the measure of $B$.

If $K$ has $C^2$ boundary, the quermassintegrals can be written in terms of principal curvatures of $K$. More precisely, denoting with $H_j$ the $j$-th normalized elementary symmetric function of the principal curvature $\kappa_1, \ldots, \kappa_{n-1}$ $\partial K$, i.e.
\[
H_0 = 1 \qquad \qquad H_j = \binom{n-1}{j}^{-1} \sum_{1 \leq i_1 < \ldots < i_j \leq n-1} \kappa_{i_1} \ldots \kappa_{i_j} \qquad j = 1,\ldots,n-1,
\]
then the quermassintegrals can be written as
\begin{equation}
    \label{quermass_con_curvature}
    W_i(K) = \frac{1}{n} \int_{\partial K} H_{i-1} \, d \mathcal{H}^{n-1} \qquad i = 1,\ldots, n.
\end{equation}
Moreover, the Steiner formula holds true also for quermassintegrals, that is
\begin{equation}
    W_j(K+\rho B) = \sum_{i=0}^{n-j} \binom{n-j}{i} W_{j+i}(K) \rho^i \qquad j=0, \ldots, n-1.
\end{equation}
For $j=1$ we have
\begin{align*}
    P(K+\rho B_1) &= n \sum_{i=0}^{n-1} \binom{n-1}{i} W_{i+1}(K) \rho^i \\
    & = P(K) + n(n-1)W_2(K)\rho +\ldots + nW_n(K) \rho^n,
\end{align*}
from which follows
\begin{equation}
    \label{derivata_perimetro}
    \lim_{\rho \to 0} \frac{P(K+\rho B) - P(K)}{\rho} = n(n-1) W_2(K),
\end{equation}
and if $\partial K$ is of class $C^2$ formula \eqref{quermass_con_curvature} gives
\[
\lim_{\rho \to 0} \frac{P(K+\rho B) - P(K)}{\rho} = (n-1) \int_{\partial K}H_1 \, d\mathcal{H}^{n-1}.
\]
Furthermore, Aleksandrov-Fenchel inequalities hold true
\begin{equation}
    \label{Aleksandrov_Fenchel_inequalities}
    \biggl( \frac{W_j(K)}{\omega_n} \biggr)^{\frac{1}{n-j}} \geq \biggl( \frac{W_i(K)}{\omega_n} \biggr)^{\frac{1}{n-i}} \qquad 0 \leq i < j \leq n-1,
\end{equation}
where equality hold if and only if $K$ is a ball. When $i=0$ and $j=1$, formula \eqref{Aleksandrov_Fenchel_inequalities} reduce to the classical isoperimetric inequality, i.e.
\[
P(K) \geq n \omega_n^{\frac{1}{n}} \abs{K}^{\frac{n-1}{n}}.
\]
In the following we will use \eqref{Aleksandrov_Fenchel_inequalities} for $i=1$ and $j=2$, that is
\begin{equation}
    \label{Aleksandrov_Fenchel_W_2}
    W_2(K) \geq n^{-\frac{n-2}{n-1}} \omega_n^{\frac{1}{n-1}} P(K)^{\frac{n-2}{n-1}}.
\end{equation}

\subsection{Some useful lemmas}
Let $\Omega \subset \R^n$ be a convex set, for $t \in [0,r_{\Omega}]$ we denote by
\[
\Omega_t = \Set{ x \in \Omega : d(x)>t } 
\]
where $d(x)$ is the distance of $x \in \Omega$ from the boundary of $\Omega$ and $r_{\Omega}$ is the inradius of $\Omega$. By the Brunn-Minkowski Theorem (\cite[Theorem 7.4.5]{Schneider_2013}) and the concavity of the distance function, the map
\[
t \mapsto P(\Omega_t)^{\frac{1}{n-1}}
\]
is concave in $[0,r_{\Omega}]$, hence absolutely continuous in $(0,r_{\Omega})$. Moreover, there exists its right derivative at $0$ and it is negative, since $P(\Omega_t)^{\frac{1}{n-1}}$ is strictly monotone decreasing.

\begin{lemma}
    \label{lemma_derivata_perimetro_1}
    Let $\Omega$ be a bounded, convex, open set in $\R^n$. Then for almost every $t \in (0,r_{\Omega})$
    \begin{equation}
    \label{eq_lemma_1}
        -\frac{d}{dt} P(\Omega_t) \geq n(n-1) W_2(\Omega_t)
    \end{equation}
    and the equality holds if $\Omega$ is a ball.
\end{lemma}
\begin{proof}
    For every $\rho \in (0,t)$ it holds
    \[
    \Omega_t + \rho B_1 \subset \Omega_{t-\rho}
    \]
    and if $\Omega$ is a ball, the two sets coincide. The monotonicity of the perimeter with respect to the inclusion of convex sets and formula \eqref{derivata_perimetro} give, for almost every $t \in (0,r_{\Omega})$,
    \begin{align*}
        -\frac{d}{dt} P(\Omega_t) & = \lim_{\rho \to 0^+} \frac{P(\Omega_{t-\rho}) - P(\Omega_t)}{\rho} \\
        & \geq \lim_{\rho \to 0^+} \frac{P(\Omega_{t} + \rho B_1) - P(\Omega_t)}{\rho} = n(n-1) W_2(\Omega_t)
    \end{align*}
\end{proof}

Combining the chain rule, the previous lemma and the fact that $\abs{\nabla d(x)} = 1$ almost everywhere, we obtain
\begin{lemma}
    \label{lemma_derivata_perimetro_2}
    Let $f \colon [0,+\infty) \to [0,+\infty)$ be a strictly increasing $C^1$ function with $f(0)=0$. Set $u(x)=f(d(x))$ and
    \[
    E_t = \Set{ x \in \Omega \, : \, u(x) > t } = \Omega_{f^{-1}(t)}
    \]
    then
    \begin{equation}
        \label{eq_lemma_2}
        -\frac{d}{dt} P(E_t) \geq (n-1) \frac{W_2(E_t)}{\abs{\nabla u}_{u=t}}
    \end{equation}
\end{lemma}

\subsection{Monotonicity of eigenfunctions}
It is well-known (by definition) that for fixed $1 <p < +\infty$ and $\Omega \subset \R^n$, the map 
\[
\beta \in \R \mapsto \lambda_{p,\beta}(\Omega)
\]
is increasing and it holds (see for instance \cite{LE_Eigenvalue})
\[
\lambda_{p,N} < \lambda_ {p,\beta} < \lambda_{p,D} \qquad \forall \beta \in (0,+\infty) 
\]
and
\[
\lim_{\beta \to 0^+} \lambda_{p,\beta} = \lambda_{p,N} = 0 \qquad \lim_{\beta \to + \infty} \lambda_{p,\beta} = \lambda_{p,D}
\]
where $\lambda_{p,N}$ and $\lambda_{p,D}$ are the first eigenvalue of the Neumann $p$-Laplacian and Dirichlet $p$-Laplacian respectively.

Now we want to observe that, if we consider a ball $B_R$ and we fix the value of the eigenfunction in the origin, the map 
\[
\beta\to v_\beta, \quad v_\beta(0)=C
\]
is decreasing, in the sense that
\[
\text{if } \beta_1 < \beta_2 \Rightarrow v_{p,\beta_1}(x) \geq  v_{p,\beta_2}(x) \qquad \forall x \in B_R
\]
First of all, let us recall that the first eigenvalue of $p$-laplacian is simple and that the corresponding eigenfunction is radially symmetric, i.e. there exists $h \colon [0,R] \to \R$ such that $u(x)=h(\abs{x})$. In the following, we will denote with $v_{p,\beta}$ both the eigenfunction and the function $h$. 

\begin{lemma}
	\label{Lemma_taglio_autofunzioni}
	Let $0<r<R$, $1<p<+\infty$, $\beta \in \R$. Let us denote by $\lambda_{p,\beta}$ the first eigenvalue of the $p$-Laplacian on the ball $B_R$ defined in \eqref{rel} with boundary parameter $\beta$ and let $v_{p,\beta}$ be the corresponding eigenfunction, then $v_{p,\beta} \rvert_{B_r}$ is the first eigenfunction of the $p$-Laplacian on the ball $B_r$ with boundary  parameter \[
	\displaystyle{ \gamma =- \frac{  \lvert v'_{p,\beta } \rvert^{p-2}(r) v'_{p,\beta}(r)}{v_{p,\beta}^{p-1}(r)} }.
	\]

	\begin{proof}
		Let us suppose $p=2$, the general case is analogous. For sake of simplicity, we denote by $\lambda_{\beta}:=\lambda_{2,\beta}$ and $v_{\beta}:=v_{2,\beta}$.
			
		By the radiality of $v_\beta$, we can infer that it is a solution to
		\[
		    \begin{cases}
		        -\Delta v_{\beta} = \lambda_{\beta} v_{\beta} & \text{in }B_r \\
		        \displaystyle{ \frac{\partial v_\beta{}}{\partial \nu} + \gamma v_{\beta} = 0 } & \text{on } \partial B_r.
		    \end{cases}
		\]
		Let us suppose by contradiction that $\lambda_{\beta}$ is not the first eigenvalue of the Robin Laplacian with boundary parameter $\gamma$. So we can choose
		\begin{equation}
		    \label{Hp_assurdo}
		    \lambda_{\gamma} < \lambda_{\beta}    
		\end{equation}
		and $w_{\gamma}$ respectively first eigenvalue and the first eigenfunction of the Robin Laplacian with boundary parameter $\gamma$, that is
		\begin{equation}
		    \label{lgam}
		     \lambda_{\gamma} = \frac{\displaystyle{ \int_{B_r} \abs{\nabla w_{\gamma}}^2 \, dx + \gamma \int_{\partial B_r} w_{\gamma}^2 \, d \mathcal{H}^{n-1} }}{\displaystyle{ \int_{B_r} w_{\gamma}^2 \, dx }}
			=
			\min_{\psi \in W^{1,2}(B_r)} \frac{\displaystyle{ \int_{B_r} \abs{\nabla \psi}^2 \, dx + \gamma \int_{\partial B_r} \psi^2 \, d \mathcal{H}^{n-1} }}{\displaystyle{ \int_{B_r} \psi^2 \, dx }}
		\end{equation}
		   
		We know that $w_{\gamma}$ is unique up to a multiplicative constant, so we can choose the constant such that $v_\beta = w_\gamma$ on $\partial B_r$.
		
		Let us consider the function
		\[
		f = \begin{cases}
		    w_\gamma(x) & \text{if }  x\in B_{r} \\
		    v_\beta(x) & \text{if } x\in B_{R} \meno B_r,
		\end{cases}
		\]
		we can use it as test function in the definition of $\lambda_{\beta}$. In particular
		\begin{align*}
		    \lambda_{\beta} & \leq \frac{ \displaystyle{ \int_{B_R} \abs{\nabla f}^2 + \beta \int_{\partial B_r} f^2 \, d\mathcal{H}^{n-1} } }{ \displaystyle{ \int_{B_R} f^2 \, dx} } \\
		    & = \frac{ \displaystyle{ \int_{B_r} \abs{\nabla w_\gamma}^2 + \int_{B_R\meno B_r} \abs{\nabla v_\beta}^2  + \beta \int_{\partial B_R} v_\beta^2 \, d\mathcal{H}^{n-1} } }
		    { \displaystyle{ \int_{B_r} w_\gamma^2 \, dx + \int_{B_R \meno B_r} v_{\beta}^2 \, dx } } .
		 \end{align*}
		 If we add and subtract $\displaystyle{\gamma\int_{\partial B_r} w_\gamma \, d\mathcal{H}^{n-1}}$ and $\displaystyle{\int_{B_r}\abs{\nabla u}^2}$, we recall that $v_\beta$ and $w_\gamma$ are eigenfunction, so \eqref{rel} hold, and we recall that $v_\beta$ and $w_\gamma$ coincides on $\partial B_r$, we have
		 \begin{align*}
		    \lambda_\beta & \leq \frac{\displaystyle{  \lambda_\gamma \int_{B_r} w_\gamma^2 \, dx + \lambda_{\beta} \int_{B_R} v_\beta^2 \, dx - \lambda_{\beta} \int_{B_r} v_\beta^2 \, dx  }}
		    {\displaystyle{ \int_{B_r} w_\gamma^2 \, dx + \int_{B_R \meno B_r} v_\beta^2 \, dx }} \\
		    & = \frac{\displaystyle{  \lambda_\gamma \int_{B_r} w_\gamma^2 \, dx + \lambda_{\beta} \int_{B_R \meno B_r} v_\beta^2 \, dx }}
		    {\displaystyle{ \int_{B_r} w_\gamma^2 \, dx + \int_{B_R \meno B_r} v_\beta^2 \, dx }}\\
		    &< \lambda_{\beta} \frac{ {\displaystyle{ \int_{B_r} w^2 \, dx + \int_{B_R \meno B_r} w^2 \, dx }} }{ {\displaystyle{ \int_{B_r} w^2 \, dx + \int_{B_R \meno B_r} w^2 \, dx }} } = \lambda_{\beta},
		\end{align*}
		where in the last formula we use \eqref{Hp_assurdo}, and this is an absurd. \qedhere
	\end{proof}
\end{lemma}

\begin{prop}
	\label{monotonia_autofunzioni}
	Let $R>0$, $1<p<+\infty$ and $\beta_1 < \beta_2$. Let us denote by $\lambda_{p,\beta_1}$ and $\lambda_{p,\beta_2}$ the eigenvalues defined in \eqref{rel} and let $v_{p,\beta_1}$ and $v_{p,\beta_2}$ be the corresponding eigenfunctions normalized such that $v_{p,\beta_1}(0) = v_{p,\beta_2}(0) > 0$, then
	\begin{equation}
	    \label{eq_monotonia_autofunzioni}
		v_{p,\beta_1} (x) \geq v_{p,\beta_2} (x) \qquad \forall x \in B_R.
	\end{equation}
\end{prop}
\begin{proof}
	Let us suppose $p=2$, the general case is analogous. For sake of simplicity, we denote by $\lambda_{\beta_i}:=\lambda_{2,\beta_i}$ and $v_{\beta_i}:=v_{2,\beta_i}$.
	
	Since both $v_{\beta_1}$ and $v_{\beta_2}$ are radial, we can write the laplacian in polar coordinates, that is
	\begin{equation}
		r^{n-1}\Delta u (r) = \bigl( r^{n-1} u'(r) \bigr) \qquad r \in [0,R].
	\end{equation}
		
	Therefore function $u_{\beta_i}$, for $i=1,2$, satisfies
	\begin{equation}
		\label{ODE_u_beta}
		v_{\beta_i}'(r) = - \frac{1}{r^{n-1}} \biggl[ \int_0^r s^{n-1} \lambda_{\beta_i} v_{\beta_i} \, ds \biggr].
	\end{equation}
		
	Since $\lambda_{\beta_1}<\lambda_{\beta_2}$ and $v_{\beta_1}(0) = v_{\beta_2}(0)$, by continuity there exists $\delta > 0$ such that
	\[
		\lambda_{\beta_1} v_{\beta_1}(s) < \lambda_{\beta_2} v_{\beta_2}(s) \qquad \forall s \in (0,\delta),
	\]
	and by \eqref{ODE_u_beta}
	\[
		v_{\beta_1}'(s) > v_{\beta_2}'(s) \qquad s \in (0,\delta).
	\]
	By classical ODE comparison result, we obtain
	\begin{equation}
		\label{confronto_u_beta}
	    	v_{\beta_1}(s) > v_{\beta_2}(s) \qquad s \in (0,\delta).
	\end{equation}
	Let us define
	\[
	A = \{ r > \delta\, : \, v_{\beta_1}(s) = v_{\beta_2}(s) \},
	\]
	we want to prove that $A$ is empty, so by \eqref{confronto_u_beta} we get the claim. Let us suppose by contradiction that $A \neq \varnothing$, hence there exists
	\[
	t = \inf A.
	\]
	By continuity $v_{\beta_1}(t) = v_{\beta_2}(t)$, and this, combined with \eqref{confronto_u_beta}, leads to
	\begin{equation}
		\label{confronto_u'_in_t}
		v_{\beta_1}'(t) < v_{\beta_2}'(t).
	\end{equation}

	Let us set
	\[
	\gamma_i = -\frac{v_{\beta_i}'(t)}{v_{\beta_i}(t)},
	\]
	by \eqref{confronto_u'_in_t} we have $\gamma_1 > \gamma_2$.
		
	\noindent By Lemma \ref{Lemma_taglio_autofunzioni}, $v_{\beta_1}$ and $v_{\beta_2}$ are the first eigenfunction of the Robin Laplacian in $B_t$ respectively with eigenvalue $\lambda_{\gamma_1}$ and $\lambda_{\gamma_2}$ with parameter $\gamma_1$ and $\gamma_2$. Therefore by monotonicity of eigenvalues with respect to the boundary parameter, we get
	\[
	\lambda_{\beta_1} = \lambda_{\gamma_1} > \lambda_{\gamma_2} = \lambda_{\beta_2}
	\]
	that is an absurd.
\end{proof}

\begin{oss}
    \label{monotonia_norme_autofunzioni}
    We highlight that, if we fix the value  $v_{p,\beta}(0)$, \eqref{eq_monotonia_autofunzioni} implies that 
    \[
    \beta \mapsto \norma{v_{p,\beta}}_{L^p}
    \]
    is non increasing, while the map
    \[
    \beta \mapsto C(n,p,\beta,\rho)= \frac{ v_{p,\beta}(0)\abs{ \Omega^{\star}} }{\norma{v_{p,\beta}}_p^p}
    \]
    is non decreasing, for all $\beta\in \R$.
\end{oss}

\section{Proofs of main results}\label{sec3}

\begin{proof}[Proof of Theorem \ref{teorema_1}]
    The quantity 
    \[
    \frac{\lambda_{p,\beta}(\Omega)-\lambda_{p,\beta}(\Omega^\star)}{\lambda_{p,\beta}(\Omega)},
    \]
    is bounded from above by 1, so inequality \eqref{beta_pos} is trivial when
    \begin{equation*}
    \label{remark}
        \left(1- \frac{n^{\frac{n}{n-1}}\omega_n^{\frac{1}{n-1}}\abs{\Omega}}{P(\Omega)^{\frac{n}{n-1}}}\right)\ge \frac{1}{C(n,p,\beta,\rho)} = \frac{ \norma{v}_{\infty}^p \abs{\Omega^{\star}} }{ \norma{v}_p^p }.
    \end{equation*}
    We can assume
    \begin{equation}
        \label{assump}
        \left(1- \frac{n^{\frac{n}{n-1}}\omega_n^{\frac{1}{n-1}}\abs{\Omega}}{P(\Omega)^{\frac{n}{n-1}}}\right)< \frac{1}{C(n,p,\beta,\rho)}.
    \end{equation}
    Let $v$ be the solution to
    \begin{equation}
        \label{palla}
        \begin{cases}
            -\Delta_p v= \lambda_{p,\beta}(\Omega^\star) \abs{v}^{p-2} v \, & \text{in} \, \, \Omega^\star\\
            \displaystyle{\abs{{\nabla v}}^{p-2}\frac{\partial v}{\partial \nu}+ \beta \abs{v}^{p-2} v=0} \, & \text{on} \, \, \partial\Omega^\star,
        \end{cases}
    \end{equation}
    it is well known that $v$ is positive and radially symmetric.
    So the function
    \[
    g(t)= \abs{\nabla v}_{v=t}
    \]
    is well defined for all $t\in (v_m, v_M)$, where $v_m=\min_{\Omega^\star} v$ and $v_M= \max_{\Omega^\star} v$.

    Let us define $u(x)=G(d(x))$, $x\in \Omega$, where

    \[
    G^{-1}(t)=\int_{v_m}^t \frac{1}{g(s)}\, ds. \qquad v_m < t <v_M.
    \]
    By construction, $u\in W^{1,2}(\Omega)$ and
    \begin{gather*}
        \min_{\Omega} u= G(0)=v_m, \\
        \norma{u}_{\infty}\le v_M, \\[1ex]
        \abs{\nabla u}_{u=t}= \abs{G'(d(x))}_{G=t}=g(t)= \abs{\nabla v}_{v=t}.
    \end{gather*}
    Let
    \[
    E_t=\{x\in \Omega: \, u(x)>t\}, \, \, \, B_t=\{x\in\Omega^\star: \, v(x)>t\}.
    \]
    By Lemma \ref{lemma_derivata_perimetro_2} and formula \eqref{Aleksandrov_Fenchel_W_2} we have
    \[
    -\frac{d}{dt}P(E_t)\ge (n-1) \frac{W_2(E_t)}{g(t)}\ge (n-1)n^{-\frac{n-2}{n-1}}\omega_n^{\frac{1}{n-1}}\frac{(P(E_t))^{\frac{n-2}{n-1}}}{g(t)}
    \]
    while for $v$ it holds
    \[
    -\frac{d}{dt}P(B_t)= (n-1) \frac{W_2(B_t)}{g(t)}= (n-1)n^{-\frac{n-2}{n-1}}\omega_n^{\frac{1}{n-1}}\frac{(P(B_t))^{\frac{n-2}{n-1}}}{g(t)}
    \]
    and $P(E_0)=P(B_0)$. Then, by classical comparison theorems for differential inequalities, 
    \begin{equation}
        \label{compper}
        P(E_t)\leq P(B_t), \quad 0\leq t \leq \norma{u}_\infty.
    \end{equation}
    Denoting by $\mu(t)=\abs{E_t}$ and by $\nu(t)=\abs{B_t}$, the coarea formula \eqref{coarea} ensures us that
    \begin{align*}
        -\mu'(t)=&\int_{u=t}\frac{1}{\abs{\nabla u}} \, d \mathcal{H}^{n-1}= \frac{P(E_t)}{g(t)}\le \frac{P(B_t)}{g(t)}\\
        =&\int_{v=t}\frac{1}{\abs{\nabla v}} \, d \mathcal{H}^{n-1}=-\nu'(t), 
    \end{align*}
    for $t\in (0,\norma{u}_\infty)$. The first equality holds true since $\abs{\nabla u} \neq 0$ in $\Set{v_m<u<\norma{u}_\infty}$ (see \cite{Brothers1988}). So the function $\nu-\mu$ is decreasing in $[0, \norma{u}_\infty]$, and
    \begin{align*}
        \int_\Omega u^p \, dx=&\int_0^{\norma{u}_\infty}p t^{p-1} \mu(t)\, dt=\int_0^{v_M}p t^{p-1} \nu(t)\, dt- \int_0^{v_M}p t^{p-1} (\nu(t)-\mu(t))\, dt \\
        \ge&\int_{\Omega^\star}v^p\, dx- v_M^p(\abs{\Omega^\star}-\abs{\Omega}).
    \end{align*}
    Moreover, by \eqref{compper}, we get
    \begin{equation*}
        \int_{u=t} \abs{\nabla u}^{p-1} \, d\mathcal{H}^{n-1}=g(t)^{p-1}P(E_t)\le g(t)^{p-1}P(B_t)=\int_{v=t} \abs{\nabla v}^{p-1}\, d\mathcal{H}^{n-1},
    \end{equation*}
    so, if we integrate from $0$ to $\norma{u}_\infty$,
    \begin{equation*}
        \int_\Omega \abs{\nabla u }^p \, dx \le \int_0^{\norma{u}_\infty} \int_{v=t} \abs{\nabla v}^{p-1}\, d\mathcal{H}^{n-1} \, dt\le \int_{\Omega^\star} \abs{\nabla v}^p.
    \end{equation*}
    We also observe that by construction both $u$ and $v$ are constant on $\partial \Omega$, so
    \begin{equation*}
        \beta \int_{\partial\Omega} u^p \, d\mathcal{H}^{n-1}= \beta v_m^pP(\Omega)=\beta\int_{\partial\Omega^\star} v^p \, d\mathcal{H}^{n-1}.
    \end{equation*}
    We finally get
    \begin{equation}
    \begin{aligned}
        \lambda_{p,\beta}(\Omega)&\le \frac{\displaystyle{\int_\Omega \abs{\nabla u }^p \, dx +\beta \int_{\partial\Omega} u^p \, d\mathcal{H}^{n-1}}}{\displaystyle{\int_\Omega u^p\, dx}} \leq 
        \frac{\displaystyle{\int_{\Omega^\star} \abs{\nabla v}^p\, dx+\beta\int_{\partial\Omega^\star} v^p \, d\mathcal{H}^{n-1}}}{\displaystyle{\int_{\Omega^\star}v^p\, dx- v_M^p(\abs{\Omega^\star}-\abs{\Omega})}} \\
        &=\lambda_{p,\beta}(\Omega^\star) \, \frac{1}{\displaystyle{1- C(n,p,\beta,\rho)\left(1-\frac{\abs{\Omega}}{\abs{\Omega^\star}}\right)}}.
    \end{aligned}
    \end{equation}
    The claim follows from \eqref{assump} as the quantity 
    \[
    1- C(n,p,\beta,\rho)\left(1-\frac{\abs{\Omega}}{\abs{\Omega^\star}}\right)
    \]
    is non-negative.
\end{proof}
\begin{oss}
    \label{rem1}
    The constant $\displaystyle{C(n,p,\beta,\rho)=\frac{\norma{v}_\infty^p\abs{\Omega^\star}}{\norma{v}_p^p}}$ depends on the perimeter of the set $\Omega$ and on $\beta$. Thanks to Proposition \ref{monotonia_autofunzioni} and Remark \ref{monotonia_norme_autofunzioni} it is possible to bound the constant $C(n,p,\beta,\rho)$ from above with a constant independent of the perimeter and of $\beta$. Indeed $\forall \beta > 0$, if we denote by $v_{p,\infty}$ the first Dirichlet eigenfunction normalized in such a way $v_{p,\beta}(0)=v_{p,\infty}(0)$, we have
    \[
    C(n,p,\beta,\rho) \leq \frac{v_{p,\infty}(0)\abs{\Omega^{\star}}}{\norma{v_{p,\infty}}_{L^p}^p} =: C(n,p)
    \]
    that is independent of the perimeter thanks to the rescaling properties of the Dirichlet $p$-Laplacian eigenfunction. 
\end{oss}

\begin{proof}[Proof of Theorem \ref{teo_beta_negative}]
    Let $v$ be a positive eigenfunction associated to $\lambda_{p,\beta}(\Omega^\star)$, then $v$ is a $p$ sub-harmonic function. We denote by $v_m=v(0)= \min_{\Omega^\star} v$ and by  $v_M= \max_{\Omega^\star} v$.

    Let us consider the function $\Tilde{v}= v_M - v$, that is a positive function with zero trace, and  $g(t)= \abs{\nabla v}_{v=t}$, $v_m <t<v_M$. 
    We set $\tilde{u}(x)= G(d(x)), \, x \in \Omega$, where $\displaystyle{G^{-1}(t) =\int_0^{t} \frac{1}{g(s)} \, ds}$ with $0<t<v_M-v_m$. By construction, $\Tilde{u} \in W^{1,p}_0(\Omega)$. 
    Now we can set $u=v_M-\Tilde{u}$, and we have:
    \begin{equation}
        \begin{gathered}
        u_M= \max_\Omega u= v_M\\
        u_m=\min_{\Omega} u = v_M- \max_{\Omega} \Tilde{u} \geq v_M- \max_{\Omega^\star} \Tilde{v}=\min_{\Omega^\star} v = v_m\\
        \abs{\nabla u}_{u=t}= \abs{\nabla v}_{v=t}=g(t) \quad u_m<t<u_M.
        \end{gathered}
    \end{equation}

    Let
    \begin{equation}
        \begin{aligned}
        \widetilde{E}_t &= \{ x \in \Omega \,:\, \tilde{u} (x)>t\}, & \widetilde{B}_t &= \{ x \in \Omega \,:\, \tilde{v} (x)>t\} ,\\
        E_t&= \{ x \in \Omega \,:\, u(x)>t\}= \Omega \setminus \overline{\widetilde{E}}_{v_M - t}, & B_t &= \{ x \in \Omega \,:\, v(x)>t\}=\Omega \setminus \overline{\widetilde{B}}_{v_M - t}.
        \end{aligned}
    \end{equation}
    By Lemma \ref{lemma_derivata_perimetro_2} and formula \eqref{Aleksandrov_Fenchel_W_2} we have
    \[
    -\frac{d}{dt}P(\widetilde{E}_t)\ge (n-1) \frac{W_2(\widetilde{E}_t)}{g(t)}\ge (n-1)n^{-\frac{n-2}{n-1}}\omega_n^{\frac{1}{n-1}}\frac{(P(\widetilde{E}_t))^{\frac{n-2}{n-1}}}{g(t)}
    \]
    while for $v$ it holds
    \[
    -\frac{d}{dt}P(\widetilde{B}_t)= (n-1) \frac{W_2(\widetilde{B}_t)}{g(t)}= (n-1)n^{-\frac{n-2}{n-1}}\omega_n^{\frac{1}{n-1}}\frac{(P(\widetilde{B}_t))^{\frac{n-2}{n-1}}}{g(t)}
    \]
    and $P(E_0)=P(B_0)$. Then, by classical comparison theorems for differential inequalities,
    \begin{equation}
        P(\widetilde{E}_t)\le P(\widetilde{B}_t), \quad 0\le t \le v_M-v_m.
    \end{equation}
    Hence
    \begin{equation}
        \label{compare_tilde}
        P(E_t)=P(\widetilde{E}_{v_M-t})\le P(\widetilde{B}_{v_M-t})=P(B_t), \quad v_m\le t \le v_M.
    \end{equation}
    Moreover, denoted by $\tilde{\mu}(t) = \lvert \widetilde{E}_t \rvert$ and  $\tilde{\nu}(t) = \lvert \widetilde{B}_t \rvert$, the coarea formula \eqref{coarea} ensures us that
    \begin{align*}
        -\tilde{\mu}'(t)=&\int_{\tilde{u}=t}\frac{1}{\abs{\nabla \tilde{u}}} \, d \mathcal{H}^{n-1}= \frac{P(\widetilde{E}_t)}{g(t)}\le \frac{P(\widetilde{B}_t)}{g(t)}\\
        =&\int_{\tilde{v}=t}\frac{1}{\abs{\nabla \tilde{v}}} \, d \mathcal{H}^{n-1}=-\tilde{\nu}'(t),  \qquad 0 \leq t < v_M-v_m.
    \end{align*}
    Moreover, setting $\mu(t)=\abs{E_t}= \abs{\Omega}- \tilde{\mu}(v_M-t)$ and by $\nu(t)=\abs{B_t}=\abs{\Omega^\star}- \tilde{\nu}(v_M-t)$,
    we have $-\mu'(t)\leq -\nu'(t)$ in $[v_m, v_M]$. So the function $\nu-\mu$ is decreasing in $[v_m, v_M]$, and
    \begin{align*}
        \int_\Omega u^p \, dx=&\int_0^{u_M}p t^{p-1} \mu(t)\, dt=\int_0^{v_M}p t^{p-1} \nu(t)\, dt- \int_0^{v_M}p t^{p-1} (\nu(t)-\mu(t))\, dt \\
        =&\int_0^{v_M}p t^{p-1} \nu(t)\, dt- \int_0^{v_m}p t^{p-1} (\nu(t)-\mu(t))\, dt-\int_{v_m}^{v_M}p t^{p-1} (\nu(t)-\mu(t))\, dt \\ \le&\int_{\Omega^\star}v^p\, dx- v_m^p(\abs{\Omega^\star}-\abs{\Omega})=\int_{\Omega^\star}v^p\, dx \left[1- \frac{v_m^p}{\norma{v}_p^p}(\abs{\Omega^\star}-\abs{\Omega})\right].
    \end{align*}
    Moreover, by \eqref{compare_tilde}, we get
    \begin{equation*}
        \int_{\tilde{u}=t} \abs{\nabla \tilde{u}}^{p-1} \, d\mathcal{H}^{n-1}=g(t)^{p-1}P(\widetilde{E}_t)\le g(t)^{p-1}P(\widetilde{B}_t)=\int_{\tilde{v}=t} \abs{\nabla \tilde{v}}^{p-1}\, d\mathcal{H}^{n-1},
    \end{equation*}
    so, if we integrate from $0$ to $\norma{u}_\infty$,
    \begin{align*}
        \int_\Omega \abs{\nabla u }^p \, dx=\int_\Omega \abs{\nabla \tilde{u} }^p \, dx =& \int_0^{\norma{\tilde{u}}_\infty} \int_{\tilde{u}=t} \abs{\nabla \tilde{u}}^{p-1}\, d\mathcal{H}^{n-1} \, dt \\
        \leq & \int_0^{\norma{\tilde{u}}_\infty} \int_{\tilde{v}=t} \abs{\nabla \tilde{v}}^{p-1}\, d\mathcal{H}^{n-1} \, dt= \int_{\Omega^\star} \abs{\nabla \tilde{v}}^p=\int_{\Omega^\star} \abs{\nabla v}^p.
    \end{align*}
    We also observe that by construction both $u$ and $v$ are constant on $\partial \Omega$, so
    \begin{equation*}
        \beta \int_{\partial\Omega} u^p \, d\mathcal{H}^{n-1}= \beta u_M^pP(\Omega)=\beta v_M^pP(\Omega)=\beta\int_{\partial\Omega^\star} v^p \, d\mathcal{H}^{n-1}.
    \end{equation*}
    We finally get
    \begin{equation}
        \begin{aligned}
            \lambda_{p,\beta}(\Omega) &\leq \frac{\displaystyle{\int_\Omega \abs{\nabla u }^p \, dx +\beta \int_{\partial\Omega} u^p \, d\mathcal{H}^{n-1}}}{\displaystyle{\int_\Omega u^p\, dx}} \leq \frac{\displaystyle{\int_{\Omega^\star} \abs{\nabla v}^p\, dx+\beta\int_{\partial\Omega^\star} v^p \, d\mathcal{H}^{n-1}}}{\displaystyle{\int_{\Omega^\star}v^p\, dx \Bigl[1- \frac{v_m^p}{\norma{v}_p^p}(\abs{\Omega^\star}-\abs{\Omega}) \Bigr]}} \\
            & = \lambda_{p,\beta}(\Omega^\star) \, \frac{1}{\displaystyle{1- \frac{v_m^p}{\norma{v}_p^p} \, (\abs{\Omega^\star}-\abs{\Omega})}}.
        \end{aligned}
    \end{equation}
    Hence, by direct calculation
    \begin{equation}
        \label{ultima_equazione}
        \frac{\lambda_{p,\beta}(\Omega^\star)-\lambda_{p,\beta}(\Omega)}{\abs{\lambda_{p,\beta}(\Omega)}} \geq \frac{v_m^p}{\norma{v}_p^p}(\abs{\Omega^\star}-\abs{\Omega}) \qedhere
    \end{equation}
\end{proof}
\begin{oss}
    \label{modbes}
    Unlike the constant in Theorem \ref{teorema_1}, the constant $\displaystyle{C(n,p,\beta,\rho)=\frac{v_m^p\abs{\Omega^\star}}{\norma{v}_p^p}}$ cannot be  bounded from below with a constant independent of the perimeter and of $\beta$. Indeed, for example if $n=p=2$ and $P(\Omega)= 2 \pi$, we have that
    \[
    v_\beta(x)= I_0\left(\sqrt{-\lambda_\beta(B_1)} \abs{x}\right),
    \]
    where $I_0$ is the modified Bessel function.
    
    We recall that for $z$ sufficiently large (see \cite[Section 9.7]{Abramowitz_Stegun}), we have
    \[
    I_0(z) \sim \frac{e^{z}}{\sqrt{2 \pi z}} \biggl[ 1 + \frac{1}{8z} + \frac{9}{2! (8z)^2} + \ldots \biggr] 
    \]
    therefore
    \[
    \lVert v_{\beta}(x) \rVert_{L^2} \sim \frac{e^{\beta}}{\beta^2} \xrightarrow{\beta \to - \infty} +\infty
    \]
    and 
    \[
    C(2,2,\beta,2\pi)= \frac{(v_\beta)_m^2 \abs{B_1}}{\norma{v_\beta}^2_2} \sim \frac{\beta^2}{e^{- \beta } } \xrightarrow{\beta \to -\infty} 0.
    \]
\end{oss}

\begin{oss}
    We want to highlight that the constant $C$ depends actually just on $n,p$ and $\rho^{\frac{p-1}{n-1}}\beta$. Indeed, for all $\Omega \subset \R^n$ bounded and convex set with $P(\Omega) = \rho$, we can consider
    \[
    \Omega_1 = \biggl(\frac{n\omega_n}{\rho}\biggr)^{\frac{1}{n-1}} \Omega \qquad \Omega_1^* = \biggl( \frac{n\omega_n}{\rho} \biggr)^{\frac{1}{n-1}} \Omega^{\star} \qquad t : = \biggl(\frac{\rho}{n\omega_n} \biggr)^{\frac{1}{n-1}}
    \]
    so $P(\Omega_1) = P(\Omega_1^{\star}) = n \omega_n$ and we have
    \begin{align*}
        \frac{ \lambda_{p,\beta} (\Omega^{\star}) - \lambda_{p,\beta} (\Omega) }{- \lambda_{p,\beta} (\Omega)} & = \frac{ \lambda_{p,\beta} ( t\Omega_1^{\star}) - \lambda_{p,\beta} (t\Omega_1) }{- \lambda_{p,\beta} (t\Omega_1)} \\
        & = \frac{ \lambda_{ p,t^{p-1}\beta} (\Omega_1^*) - \lambda_{ p,t^{p-1}\beta} (\Omega_1) }{\lambda_{ p,t^{p-1}\beta} (\Omega_1)} \\
        & \geq C\bigl( n,p, n \omega_n, t^{p-1}\beta \bigr) \biggl(1 - \frac{n^{\frac{n}{n-1}}\omega_n^{\frac{1}{n-1}}\abs{\Omega}}{P(\Omega)^{\frac{n}{n-1}}}\biggr) \\
        & = C\Bigl( n,p, n \omega_n, \Bigl(\frac{\rho}{n\omega_n} \Bigr)^{\frac{p-1}{n-1}}\beta \Bigr) \biggl(1 - \frac{n^{\frac{n}{n-1}}\omega_n^{\frac{1}{n-1}}\abs{\Omega}}{P(\Omega)^{\frac{n}{n-1}}}\biggr) \\
        & = C \bigl( n,p,\rho^{\frac{p-1}{n-1}}\beta \bigr) \biggl(1 - \frac{n^{\frac{n}{n-1}}\omega_n^{\frac{1}{n-1}}\abs{\Omega}}{P(\Omega)^{\frac{n}{n-1}}}\biggr)
    \end{align*}
\end{oss}

\begin{proof}[Proof of Theorem \ref{teo_quantitativa_con_Hausdorff}]
    Using \eqref{ultima_equazione}, the isoperimetric inequality and
    \[
    \abs{\lambda_{p,\beta}(\Omega)} \geq \abs{\beta} \frac{P(\Omega)}{\abs{\Omega}}, 
    \]
    we have
    \begin{equation}
        \label{mod_beta_neg}
        \begin{aligned}
            \lambda_{p,\beta}(\Omega^\star)-\lambda_{p,\beta}(\Omega) & \geq\abs{\lambda_{p,\beta}(\Omega)}  \frac{v_m^p}{\norma{v}_p^p}(\abs{\Omega^\star}-\abs{\Omega}) \\
            & \geq \abs{\beta} \frac{P(\Omega)}{\abs{\Omega}}\frac{v_m^p}{\norma{v}_p^p}(\abs{\Omega^\star}-\abs{\Omega}) \\
            & \geq \abs{\beta} \frac{ n^{\frac{n}{n-1}} \omega_n^{\frac{1}{n-1}} }{\rho^{\frac{1}{n-1}}}\frac{v_m^p}{\norma{v}_p^p}(\abs{\Omega^\star}-\abs{\Omega}).
        \end{aligned}
    \end{equation}
    Now, if we suppose that $\lambda_{p,\beta}(\Omega^\star)-\lambda_{p,\beta}(\Omega) \leq \delta_0$, by \eqref{mod_beta_neg} we have
    \[
    \abs{\Omega^\star}-\abs{\Omega} \leq K(n,\rho,p,\beta) \delta_0.
    \]
    So by Lemma \ref{lemma_fugl_mod} we conclude
    \[
    \lambda_{p,\beta}(\Omega^\star)-\lambda_{p,\beta}(\Omega) \geq C(n,\rho, p , \beta) \, g(\mathcal{A}_{\mathcal{H}}^{\star} (\Omega) ). \qedhere
    \]
\end{proof}

\begin{oss}
    Our result, Theorem \ref{teo_quantitativa_con_Hausdorff}, applies only when \[
    \lambda_{p,\beta}(\Omega^\star)-\lambda_{p,\beta}(\Omega)\le \delta_0.
    \]
    It is possible to get rid of this constraint, obtaining a weaker result.
    
    In order to obtain it, we need the quantitative version of the isoperimetric inequality proved in \cite{Fusco_Maggi_Pratelli}
    \begin{equation}
        \label{fmp}
        P(\Omega) \geq P(\Omega^{\sharp}) \bigl( 1 + \gamma(n) \alpha(\Omega)^2 \bigr) \qquad \text{where } \alpha(\Omega) = \min \Set{ \frac{\abs{\Omega \triangle B_r}}{\abs{\Omega}} \, | \, \abs{B_r} = \abs{\Omega}},
    \end{equation}
    and the following result \cite[Lemma 4.2]{Esposito_Fusco_Trombetti}: there exists a constant $C(n)$ such that if $C,W$ are open and convex sets such that $\abs{C}=\abs{W}$ and $\abs{C \triangle W} < \frac{\abs{C}}{2}$, it holds
    \begin{equation}
        \label{eq_Trombetti}
        d_{\mathcal{H}}(C,W) \leq C(n) \bigl[ \diam(C) + \diam(W) \bigr] \biggl( \frac{\abs{C \triangle W}}{\abs{C}} \biggr)^{\frac{1}{n}}.
    \end{equation}
    Moreover we have to recall that
    \[
    \abs{\Omega^{\star}} = \frac{P(\Omega)^{\frac{n}{n-1}}}{n^{\frac{n}{n-1}} \omega_n^{\frac{1}{n-1}} } \qquad \text{ and } \qquad P(\Omega^{\sharp}) = n \omega_n^{\frac{1}{n}} \abs{\Omega}^{\frac{n-1}{n}}
    \]
    and we have
    \begin{align*}
        1-\frac{\abs{\Omega}}{\lvert \Omega^{\star} \rvert } & = 1 - \frac{n^{\frac{n}{n-1}} \omega_n^{\frac{1}{n-1}} \abs{\Omega}}{P(\Omega)} \\
        &\geq  1 - \frac{1}{\displaystyle{\bigl( 1+ \gamma (n) \alpha^2(\Omega) \bigr)^{\frac{n}{n-1}} }} \\
        & \geq 1 - \frac{1}{\displaystyle{\bigl( 1+\gamma(n) \alpha^2(\Omega)  \bigr) }} \\[1ex]
        & = \frac{\gamma(n) \alpha^2(\Omega)}{1+\gamma(n) \alpha^2(\Omega)},
    \end{align*}
    where we used Bernoulli's inequality
    \[
    (1+x)^r \geq 1+rx \qquad \forall x \geq - 1, \, \forall r \geq 0.
    \]
    Since $0 < \alpha^2(\Omega) < 4$ we have
    \[
    1-\frac{\abs{\Omega}}{\lvert \Omega^{\star} \rvert } \geq \frac{\gamma(n)}{1+4 \gamma(n)} \, \alpha^2(\Omega) = C(n) \alpha^2(\Omega).
    \]
    If $\Omega^{\sharp}$ is a ball that realizes the minimum in \eqref{fmp}, then using \eqref{eq_Trombetti} we obtain
    \[
    C(n) \alpha^2(\Omega) \geq C(n) \,  \frac{\displaystyle{ d_{\mathcal{H}}(\Omega,\Omega^{\sharp})^{2n} } }{\displaystyle{ \bigl[ \diam(\Omega) + \diam(\Omega^{\sharp}) \bigr]^{2n} }}
    \]
    and so
    \[
    \frac{ \lambda_{p,\beta} (\Omega^{\star}) - \lambda_{p,\beta} (\Omega) }{\abs{ \lambda_{p,\beta} (\Omega)}}\ge C(n, p, \beta \rho) C(n)  \frac{\displaystyle{ d_{\mathcal{H}}(\Omega,\Omega^{\sharp})^{2n} } }{\displaystyle{ \bigl[ \diam(\Omega) + \diam(\Omega^{\sharp}) \bigr]^{2n} }}.
    \]
\end{oss}

\addcontentsline{toc}{chapter}{Bibliografia}
\bibliographystyle{alpha}
\bibliography{biblio}

\renewcommand{\abstractname}{}
\begin{abstract}

    \noindent 	\textsc{Dipartimento di Matematica e Applicazioni “R. Caccioppoli”, Università degli Studi di Napoli “Federico II”, Complesso Universitario Monte S. Angelo, via Cintia - 80126 Napoli, Italy.}
	
	\textsf{e-mail: vincenzo.amato@unina.it}
	
	\textsf{e-mail: albalia.masiello@unina.it}
	\vspace{0.5cm}
	
	\noindent \textsc{Mathematical and Physical Sciences for Advanced Materials and Technologies, Scuola Superiore Meridionale, Largo San Marcellino 10, 80126 Napoli, Italy.}
	
	\textsf{e-mail: andrea.gentile2@unina.it}
	
\end{abstract}

\end{document}